\documentclass[12pt]{article}
\usepackage{amssymb,amsthm,latexsym,amsmath,amscd}

\newtheorem{theorem}{Theorem}

\newtheorem{corollary}{Corollary}
\newtheorem*{theorema}{Theorem}

\newtheorem*{mainlemma}{Main Lemma}

\theoremstyle{definition}

\theoremstyle{definition}

\theoremstyle{definition}

\theoremstyle{definition}

\theoremstyle{remark}

\theoremstyle{definition}

\theoremstyle{definition}

\theoremstyle{definition}

\theoremstyle{definition}
\newtheorem{step}{Step}

\theoremstyle{definition}

\theoremstyle{definition}
\newtheorem*{conjecturea}{Conjecture}

\newcommand{\Dom}{\operatorname{Dom}}

\newcommand{\Real}{\operatorname{Re}}

\newcommand{\dee}{\partial}
\newcommand{\dba}{\bar{\partial}}
\newcommand{\ddba}{\partial \bar{\partial}}

\newcommand{\ipa}[1]{\partial_{#1}}
\newcommand{\pa}[2]{\frac{\partial{#1}   } {\partial{#2}            } }

\newcommand{\deltanto}[2]{\delta^{-\frac{#1}{#2}}}

\begin{document}
\title{Monotonicity of Subelliptic Estimates \\ on Rigid Pseudoconvex Domains}
\author{Jae-Seong Cho\\
Department of Mathematics\\
Purdue University\\
cho1@math.purdue.edu}
\date{\today}
\maketitle

\begin{abstract}
In this paper we will present  monotonicity of
subelliptic estimates involving Levi forms on rigid pseudoconvex domains. As an application
of monotonicity, we will show that if a rigid domain is given by a
sum of the squares of monomials and if the domain is of finite type, then the sharp
subelliptic estimate of this domain equals the reciprocal of the D'Angelo's $1$-type.

\end{abstract}

\section{Introduction}
This paper mainly concerns  the following  class of domains. Let $\Omega$ be a pseudoconvex domain in $\mathbb
C^{n+1}$ and let the origin be  a boundary point of $\Omega$. Following the terminology in \cite{BRT}, we say
that $\Omega$ is a \emph{rigid pseudoconvex domain} near the origin if there exists a complex coordinate
$(z_1,\dots,z_{n+1})$ such that  $\Omega$ is defined near the origin by
\begin{equation}\label{E:Defining}
\Omega = \{(z_1,\dots,z_{n+1}) \in \mathbb C^{n+1} \;: \;\Real z_{n+1} +u(z_1,\dots,z_n)<0 \},
\end{equation}
where $u$ is a plurisubharmonic function depending only on the first $n$-variables with $u(0)=0$.
We say that $u$ is the \emph{mixed term} of the boundary
defining function of $\Omega$ near the origin.

Let $\Omega$ be
a smoothly bounded pseudoconvex domain in $\mathbb C^{n+1}$ and let $p_0$ be a
point on the boundary of $\Omega$.
We say that  \emph{subelliptic estimate of order $\epsilon
>0$}  holds on $(0,1)$-forms at $p_0$ if there exist a constant $C$ and a
neighborhood $U$ of $p_0$ such that the estimate
\begin{equation}\label{E:SubEllipEst}
|||\varphi|||_{\varepsilon}^2 \leq C(||\bar{\partial}\varphi||^2 +
||\bar{\partial}^{\ast}\varphi||^2+||\varphi||^2)
\end{equation}
is valid for all  $\varphi \in \mathcal{D}^{0,1}(U)$. Here $|||\cdot|||_{\epsilon}$ denotes the tangential
Sobolev norm of order $\epsilon$ and $\mathcal{D}^{0,1}(U)$ refers to the set of smooth $(0,1)$-forms $\varphi$
in $\Dom(\dba) \cap \Dom(\dba^{\ast})$. The \emph{D'Angelo's $1$-type} of the boundary of $\Omega$ at $p_0$ is
defined by
\begin{equation}
T(b\Omega,p_0)= \sup_{z} \frac{v(z^{\ast}r)}{v(z)}.
\end{equation}
Here $z$ runs over the set of all germs of parameterized complex analytic curves at $p_0$
and $v(\cdot)$ denotes the vanishing order. Even though the type function, $p \to
T(b\Omega,p)$, is neither upper nor lower semi-continuous, D'Angelo showed that if
$T(b\Omega,p_0)$ is finite, then it  is locally bounded near $p_0$ \cite{DA82}.

Catlin showed that  if $\Omega$ is a smoothly bounded pseudoconvex domain in $\mathbb C^{n+1}$ and $p_0 \in
b\Omega$, then $T(b\Omega,p_0)$ is finite if and only if there is  a subelliptic estimate of some order $\epsilon
>0$ on $(0,1)$-forms at $p_0$. More precisely, for necessary condition of
subellipticity  he \cite{Ca83} showed that if $T(b\Omega,p_0)$ is finite and a subelliptic estimate of order
$\epsilon$ of the form \eqref{E:SubEllipEst} holds, then $\epsilon$ must satisfy
%\begin{equation}\label{E:CatIneq}
$\epsilon \leq \frac{1}{T(b\Omega,p_0)}$.
%\end{equation}
For sufficiency he reduced the problem to construct a family of plurisubharmonic functions
$\{\lambda_{\delta}\}$ with large Hessians near the boundary.

\begin{theorema}[{Catlin \cite[Theorem 2.2]{Ca87}}]\label{T:Ca87}
Let $p_0$ be a point on the boundary of a smoothly bounded pseudoconvex  domain in $\mathbb C^{n+1}$. Let
$\Omega_{\delta}=\{z \mid  r(z) <\delta \}$ and let $S_{\delta}=\{ z \mid -\delta <r(z) <\delta \}$. Suppose that
there exist a neighborhood $\tilde U$ of $p_0$, a parameter $\epsilon$ with $0<\epsilon\leq \frac{1}{2}$, and a
constant $c>0$ such that for any sufficiently small $\delta>0$ there exists a smooth plurisuharmonic function
$\lambda_{\delta}$ in $\tilde U \cap \Omega_{\delta}$ satisfying the following properties:
\begin{align}
& \textup{(i)} \quad |\lambda_{\delta}| \leq 1 \quad \mbox{on $\tilde U \cap \Omega_{\delta}$}
\label{E:uniformbounded}\\
& \textup{(ii)} \quad \sum_{i,j=1}^{n+1} \frac{\partial^2 \lambda_{\delta}}{\partial z_i
\partial \bar z_j} s_i \bar{s}_j \geq c \delta^{-2\epsilon}\sum_{i=1}^{n+1}|s_i|^2, \quad
s_i \in \mathbb C,  \quad\mbox{on \;$\tilde U\cap S_{\delta}$}. \label{E:PSH}
\end{align}
Then there exists a neighborhood $\tilde U'$ of $p_0$ with $\tilde U'\subset \subset \tilde
U$ such that a subelliptic estimate of order $\epsilon$ holds in $\tilde U'$.
\end{theorema}

In this paper we are primarily interested in a construction of a family of plurisubharmonic functions,
$\{\lambda_{\delta}\}$, satisfying \eqref{E:uniformbounded} and \eqref{E:PSH}, for a rigid pseudoconvex domain
defined by \eqref{E:Defining}. The main idea of this paper is to resolve $\ddba u$ at degenerated points by
adding small cutoff functions in a uniform way.
 Since the mixed term $u$ is independent of
the last variable $z_{n+1}$, there is a natural projection from the boundary of a rigid domain to a neighborhood
of the origin in $\mathbb C^n$. Under this projection one can interpret the Levi form of the boundary of $\Omega$
in terms of $\ddba u$. This observation enables us to derive the following lemma, a modified version of
\cite[Theorem 2.2]{Ca87} for rigid pseudoconvex domains.
%%%
%%%
%%% Main Lemma
%%%
%%%
\begin{mainlemma}\label{L:Main}
Let $u(z)$ be a smooth plurisubharmonic function in a neighborhood $V$ of the origin in
$\mathbb C^{n}$ with $u(0)=0$ and let $\Omega$ be a rigid pseudoconvex domain defined by
\eqref{E:Defining} near the origin in $\mathbb C^{n+1}$. Suppose that there exist a
neighborhood $U \subset \subset V$ of the origin in $\mathbb C^n$, a constant $C>0$, and a
parameter $\epsilon$ with $0< \epsilon \leq  \frac{1}{2}$ so that for each sufficiently small $\delta
>0$ there exists a smooth function $\rho_{\delta}(z)$ on $U$ with the following properties:
\begin{itemize}
\item [\textup{(i)}] $0 \leq \rho_{\delta}(z)  \leq 1$ for all $z \in U$,

\item [\textup{(ii)}] for all $z\in U$ and for all $L=s_1\pa{}{z_1} + \dots + s_n
\pa{}{z_n}$, $s_i \in \mathbb C$
\begin{equation}\label{E:resolve}
\ddba \left( \frac{u}{\delta} + \rho_{\delta}\right) (L,\bar{L})(z) \geq
C\delta^{-2\epsilon} |L|^2.
\end{equation}
\end{itemize}
Then  a subelliptic estimate of order $\epsilon$ holds at the origin for $\Omega$.
\end{mainlemma}
%%%
%%%
%%% Monotone Property
%%%
%%%

We will give the proof of Main Lemma in Section \ref{S:MainLemma}. As an immediate consequence of Main Lemma, we
obtain the following monotone property of subelliptic estimates on two rigid domains.
%%%
%%%
%%% Theorem of Monotonicity
%%%
\begin{theorem}\label{T:Mono}
Let $u_1$ and $u_2$ be plurisubharmonic functions defined on a neighborhood $V$ of the origin in $\mathbb C^n$
with $u_1(0)=u_2(0)=0$. Let $\Omega_1$ and $\Omega_2$ be rigid pseudoconvex domains with mixed terms, $u_1$ and
$u_2$, respectively, near the origin in $\mathbb C^{n+1}$. Suppose that there exist a neighborhood $U \subset
\subset V$ in $\mathbb C^n$, a constant $C>0$, and a parameter $\epsilon$ with $0<\epsilon\leq \frac{1}{2}$ such
that for each sufficiently small $\delta
>0$ there exists a smooth function $\rho_{\delta}(z)$
satisfying \textup{(i)} and \textup{(ii)}  for $u_2$ in Main Lemma. If for all $z \in V$,
\begin{equation}
\ddba u_1(L,\bar L)(z) \geq \ddba u_2(L,\bar L)(z), \quad \mbox{$L=\sum_{i=1}^n s_i
\pa{}{z_i}$},
\end{equation}
then a subelliptic estimate of order $\epsilon$ holds at the origin for both $\Omega_1$ and
$\Omega_2$.
\end{theorem}
%%%
%%%
%%%

In Section \ref{S:Monomial}, as an application of the monotone property,  we will consider the largest
subelliptic gain of
 a rigid pseudoconvex domain whose Levi form is greater than or equal
to the one of a diagonal domain (Theorem \ref{T:AlmDiag}). As a corollary of Theorem
\ref{T:AlmDiag}, We will provide an answer to the D'Angelo's conjecture when
the mixed term of a rigid pseudoconvex domain is a sum of
the squares of monomials (Corollary \ref{C:Monomial}).
\begin{conjecturea}[D'Angelo \cite{DA93}]
Let $\mathcal O_n$ be the ring of germs of holomorphic functions at the origin $\mathbb
C^n$ and let $I$ be an ideal generated by $f_1,\dots,f_l \in \mathcal O_n$ with $f_j(0)=0$
$j=1,\dots,l$. Let $m(I)$ denote  $\dim_{\mathbb
C}\mathcal O_n/I$. Let $\Omega \subset \subset \mathbb C^{n+1}$ be a rigid domain whose
boundary near the origin is defined by
%\begin{equation}
$r(z')=2\Real z_{n+1} + \sum_{i=1}^l |f_j(z)|^2$,
%\end{equation}
where $z=(z_1,\dots,z_n)$ and $z'=(z_1,\dots,z_{n+1})$. If $m(I)$ is finite, then a
subelliptic estimate holds for
\begin{equation}
\frac{1}{2m(I)} \leq \epsilon \leq \frac{1}{T(b\Omega,0)}.
\end{equation}
\end{conjecturea}

\medskip
\noindent \emph{Acknowledgement}. I am greatly indebted to David W. Catlin for sharing his
original idea. He has generously put his energy and time into explaining his deep theory
to me ever since I came to Purdue University as a post-doc at 2006. Without this long
discussion and inspiration this paper would not come out. I also would like to express my
thanks to John P. D'Angelo for encouraging my work constantly. Furthermore, he kindly
pointed out his recent preprint about the effectiveness of Kohn's algorithm, which was
very helpful to finish this work.
%%%
%%%
%%% Begin Section of Main Theorem
%%%
%%%
\section{Proof of Main Lemma}\label{S:MainLemma}

Let $\Omega \subset \mathbb C^{n+1}$ be  a rigid domain whose boundary is defined near the origin by
\begin{equation}
r(z')=2\Real z_{n+1} + u(z),
\end{equation}
where $u(z)$ is a smooth plurisubharmonic function, depending only on $z=(z_1,\dots,z_n)$, with $u(0)=0$.
Let  $\tilde U$ be an open set in $\mathbb C^{n+1}$ defined by
$ \tilde U = U \times \mathbb C  \subset \mathbb C^{n+1}$,
and let
\begin{align}
\Omega_{\delta} &= \{z' \in \mathbb C^{n+1} \mid r(z') < \delta \}\\
S_{\delta} &=\{z' \in \mathbb C^{n+1} \mid -\delta < r(z') \leq 0 \}.
\end{align}
For simplicity, we set
\begin{equation}
\ipa{i} = \mbox{$\pa{}{z_i}$} \quad i=1,\dots, n+1.
\end{equation}
Let denote a constant vector field in $\mathbb C^n$ by
\begin{equation}
L=\sum_{i=1}^ns_i\ipa{i} \quad  s_i \in \mathbb C.
\end{equation}
Consider $(1,0)$-vector fields $L_1,\dots,L_{n+1}$ on $ \tilde U$, defined by
\begin{align}
L_i\hfill \;&= \mbox{$\ipa{i} - \frac{r_{z_i}}{r_{z_{n+1}}}\;\ipa{n+1} = \ipa{i} -u_{z_i}\ipa{n+1}$},
\quad i=1,\dots,n\\
 L_{n+1} &= \ipa{n+1}.\notag
\end{align}
Note that $L_ir=0$ with $i=1,\dots,n$ for all $z' \in \tilde U$.  Let us denote
\begin{equation*}
L'  =\sum_{i=1}^{n+1}s_iL_i, \quad \tilde L  =\sum_{i=1}^{n+1}s_i \ipa{i}, \quad  s_i \in \mathbb C.
\end{equation*}
Since $u_{z_i}$ is bounded on $U \subset \subset V$ for $i=1,\dots,n$, it follows that
there exist constants $C_1,C_2>0$, depending only on $U$, so that
\begin{equation}\label{E:NormEqui}
C_1|L'|^2 \leq \sum_{i=1}^{n+1} |s_i|^2 \leq C_2|L'|^2.
\end{equation}
%%%
%%%
%%% First Step of Main Lemma
%%%
%%%
\begin{step} \label{S:StrPos}
Let $\tilde \lambda_{\delta}$ be a function defined on $V \times \mathbb C \subset \mathbb
C^{n+1}$  by
\begin{equation}
\tilde \lambda_{\delta}=e^{\frac{r}{\delta}} + e^{-3}\rho_\delta.
\end{equation}
If $z' \in U \times \mathbb C$ with $r(z') \geq -3\delta$, then for all $L'$ at
$z'$,
\begin{equation}\label{E:StrPos}
\ddba \tilde \lambda_{\delta}(L',\bar L')(z') \gtrsim \delta^{-2\epsilon}|L'|^2.
\end{equation}
\end{step}
%%%
%%%
%%% Proof of First Step of Main Lemma
%%%
%%%
\begin{proof}[Proof of Step \ref{S:StrPos}]
Note  that
\begin{equation}\label{E:ddbatildelambda}
 \ddba(\tilde \lambda_{\delta}) =
e^{\frac{r}{\delta}} \left[\frac{\dee r \wedge \dba r}{\delta^2} +   \frac {\ddba
u}{\delta}\right] + \frac{ \ddba \rho_{\delta}}{e^3}.
\end{equation}
Since
\begin{align}
&\ddba u(L',\bar{L}')(z')= \ddba u(L,\bar L) (z) \notag\\
&\ddba \rho_{\delta} (L',\bar L')(z') = \ddba \rho_{\delta}(L,\bar L)(z), \label{E:project}
\end{align}
it follows that for any $z'=(z,z_{n+1}) \in \tilde U$ with $r(z')\geq -3\delta $,
\begin{align}
\ddba \tilde \lambda_{\delta} (L',\bar L')(z') & = e^{\frac{r}{\delta}} \left[ \frac{(\dee
r
\wedge \dba r)}{\delta^2}(L',\bar L')(z')\right]\notag\\
&\;\; + e^{\frac{r}{\delta}}\left[\frac{\ddba u}{\delta} (L,\bar L)(z)\right] + \frac{
\ddba \rho_{\delta}}{e^3}(L,\bar L)(z) \notag\\
&= e^{\frac{r}{\delta}} \left[\frac{|s_{n+1}|^2}{\delta^2}\right]
+e^{\frac{r}{\delta}}\left[\frac{\ddba u}{\delta} (L,\bar L)(z)\right] +\frac{
\ddba \rho_{\delta}}{e^3}(L,\bar L)(z) \notag\\
&\geq e^{-3}\left[\frac{|s_{n+1}|^2}{\delta^2}  + \ddba \left( \frac{u}{\delta} +
\rho_{\delta}\right) (L,\bar{L})(z)\right]\notag\\
& \geq e^{-3}\left[\frac{|s_{n+1}|^2}{\delta^2} +C\delta^{-2\epsilon} |L|^2\right].\label{E:longineq}
\end{align}
In fact, the first equality follows from  \eqref{E:ddbatildelambda} and \eqref{E:project}. The second equality
results from the fact that $ L_ir=0$ for $1\leq i \leq n$ and the third inequality is obtained by the fact that
$e^{\frac{r(z')}{\delta}} \geq e^{-3}$ for  $r(z') \geq -3\delta$. The fourth inequality follows from
\eqref{E:resolve}. Since $0<\epsilon \leq \frac{1}{2}$ and $\delta$ is small, it follows from \eqref{E:NormEqui}
and \eqref{E:longineq} that \eqref{E:StrPos} holds.
\end{proof}
%%%
%%%
%%% Second Step for main Lemma
%%%
%%%
\begin{step}\label{S:Second}
Let $p(t)$ be an increasing smooth convex function with $p(t)=0$ for $t \leq e^{-2}$, and $p(t)>0$, $p'(t)>0$,
for  $t>e^{-2}$. Let
\begin{equation}
\lambda_{\delta}(z')=p \circ \tilde \lambda_{\delta}(z').
\end{equation}
Then $\lambda_{\delta}$ is a plurisubharmonic function on $\tilde U$.
\end{step}
%%%
%%%
%%% proof of Second step of main Lemma
%%%
%%%
\begin{proof}[Proof of Step \ref{S:Second}]
Since $\tilde \lambda_{\delta}=e^{\frac{r}{\delta}} + e^{-3} \rho_{\delta}$ and $0 \leq
\rho_{\delta} \leq 1$, it follows that
\begin{equation}
\tilde \lambda_{\delta}(z') \leq e^{-3} + e^{-3} < e^{-2}, \quad \mbox{for $z' \in \tilde U$
with $r(z') \leq -3\delta$.}
\end{equation}
Since $p(t)=0$ for $t \leq e^{-2}$, we have
\begin{equation}\label{E:Zero}
\lambda_{\delta}(z')=p(\tilde \lambda_{\delta}(z'))=0, \quad \mbox{for $z' \in \tilde U$
with $r(z') \leq -3\delta$.}
\end{equation}
If  $z' \in \tilde U$ with $r(z') \geq -3\delta$, then since $\tilde \lambda_{\delta}(z')$
is plurisubharmonic from Step \ref{S:StrPos} and since $p(t)$ is a convex increasing
function, it follows that $\lambda_{\delta}$ is plurisubharmonic on $z' \in \tilde U$ with
$r(z') \geq -3\delta$. Therefore, combining with \eqref{E:Zero}, we conclude that
$\lambda_{\delta}$ is plurisubharmonic on $\tilde U$.
\end{proof}
%%%
%%%
%%% Continue the proof of Main Lemma
%%%
%%%
\noindent \emph{Continue the proof of Main Lemma \ref{L:Main}.} Since $\tilde \lambda_{\delta}(z') \leq e+e^{-3}$
for $z' \in \tilde U \cap \Omega_\delta$ and since $\lambda(z')=p \circ \tilde \lambda_{\delta}(z')=0$ for $z'
\in \tilde U$ with $r(z') <-3\delta$ by \eqref{E:Zero}, it follows that there exists a constant $C'>0$
independent of $\delta$ such that
\begin{equation}
|\lambda_{\delta}(z')| \leq C', \quad \mbox{for all} \quad z' \in \tilde U \cap
\Omega_{\delta}.
\end{equation}
Since $p$ is a smooth convex increasing  function with $p'(t)>0$ for $t>e^{-2}$, there exists a constant $c>0$ so
that
\begin{equation}\label{E:PStr}
p'(t) \geq c, \quad \mbox{for} \quad e^{-1} \leq t \leq e+e^{-3}.
\end{equation}
Since
\begin{equation}
e^{-1} \leq \tilde \lambda_{\delta}(z') \leq e+e^{-3} \quad \mbox{for $z' \in \tilde U \cap S_{\delta}$},
\end{equation}
it follows from \eqref{E:StrPos} and \eqref{E:PStr} that there exists a constant $C''>0$
independent of $\delta$ such that for all $z' \in S_{\delta} \cap \tilde U$ and all $L'$ at
$z'$,
\[
\ddba \lambda_{\delta} (L',\bar L')(z') \geq C''\delta^{-2\epsilon} |L'|^2.
\]

%%%
%%%
%%% begin Section of Monomial domains
%%%
%%%
\section{Sharp Subelliptic Estimates on Rigid Monomial Domains}\label{S:Monomial}
%%%
%%%
%% Proposition for almost diagonalizable domains
%%%
%%%
In this section we will consider subellipticity of a rigid domain whose Levi form is
greater than or equal to the one of a diagonal domain. As an application of monotonicity,
we will show that if the mixed term of the boundary of a rigid domain is given by sum of
the squares of monomials and it is of finite type at the origin, then the sharp subelliptic
estimate of this domain is the reciprocal of the $1$-type.

%%%
%%%
%%% Theorem for almost diagonalizable
%%%
%%%
\begin{theorem}\label{T:AlmDiag}
Let $\Omega$ be a rigid pseudoconvex domain near the origin and let $u$ be the mixed term of the boundary of
$\Omega$ near the origin.
Let $m_1,\dots,m_n$ be positive integer.  Suppose that there exists a neighborhood $V$ of
the origin in $\mathbb C^n$ so that for all $z \in V$
\begin{equation}\label{E:AlmDiag}
\ddba u (L,\bar L) (z) \geq \ddba \left(\sum_{i=1}^n |z_i|^{2m_i} \right) (L,\bar L)
\quad \mbox{for} \quad L=\sum_{i=1}^n s_i \ipa{i}.
\end{equation}
Then a subelliptic estimate at the origin for $\Omega$  holds  of order
\begin{equation}
\epsilon = \frac{1}{2 \max\{m_i : 1 \leq i \leq n\}}.
\end{equation}
\end{theorem}
%%%
%%%
%%%  proof for almost diagonalizable domains
%%%
%%%
\begin{proof}
By Theorem \ref{T:Mono} it is enough to show that
there exist a neighborhood $U$ of the origin in $\mathbb C^n$, a constant $C>0$, and
a parameter $\epsilon$ with $0<\epsilon \leq \frac{1}{2}$ such that for each sufficiently small $\delta>0$
there exists a smooth function $\rho_{\delta}$ satisfying \textup{(i)} and \textup{(ii)}
for $\sum_{i=1}^n |z_i|^{2m_i}$ in Main Lemma.

Let $\chi$ be a smooth function such that
$0 \leq \chi(t) \leq 1 $ for $t\geq 0$,
$\chi(t)=t$ for $ t \leq \frac{1}{2}$,
and $\chi(t)  =0 $ for $ t\geq 1$.
Let us denote
$m=\max\{m_i : 1 \leq i \leq n \}$.
For each small $\delta >0$ define
$\tau_i(\delta)=\delta^{\frac{1}{2m_i}}$, $1\leq i \leq n$ and
\begin{equation}
\rho_{\delta}(z) = c \sum_{i=1}^n \chi \left( \frac{|z_i|^2}{(\tau_i(\delta))^2} \right),
\end{equation}
where $c>0$ is to be determined. Let $U \subset \subset V$ be neighborhood of the origin.
For $L= \sum_{i=1}^n s_i\ipa{i}$ and $z \in U$,
\begin{align}
&\ddba \left(\frac{\sum_i^n |z_i|^{2m_i}}{\delta} + \rho_{\delta}\right)(L,\bar L)(z)\notag\\
&= \ddba \left( \frac{\sum_{i=1}^n |z_i|^{2m_i}}{\delta} + c \sum_i^n \chi\left(
\frac{|z_i|^2}{(\tau_i(\delta))^2}\right) \right)(L,\bar L)(z) \notag \\
&=\sum_{i=1}^n a_i|s_i|^2,
\end{align}
where for $i=1,\dots,n$,
\begin{equation}\label{E:Coea}
a_i =\frac{|z_i|^{2m_i-2}}{\delta} + c \chi' \left(\frac{|z_i|^2}{(\tau_i(\delta))^2}\right)
\frac{1}{(\tau_i(\delta))^2} + c \chi''\left( \frac{|z_i|^2}{(\tau_i(\delta))^2}
\right)\frac{|z_i|^2}{(\tau_i(\delta))^4}.
\end{equation}
Since $a_i$ depends only on the $i$-th variable $z_i$, it is enough to estimate
the constant $a_i$ for each variable $z_i$. If $|z_i|^2 \geq (\tau_i(\delta))^2$, then since the last
two terms in \eqref{E:Coea} vanish, it follows that
\begin{equation}\label{E:Big}
a_i \geq \frac{(\tau_i(\delta))^{2m_i-2}}{\delta}=\delta^{-\frac{1}{m_i}}, \quad 1\leq i \leq n.
\end{equation}
If $|z_i|^2 \leq \frac{1}{2} (\tau_i(\delta))^2$, then since $\chi'(\frac{|z_i|^2}{(\tau_i(\delta))^2})=1$ and
$\chi''(\frac{|z_i|^2}{(\tau_i(\delta))^2})=0$, we have
\begin{equation}\label{E:Small}
a_i \geq c \frac{1}{(\tau_i(\delta))^2}=c \delta^{-\frac{1}{m_i}}, \quad 1 \leq i \leq n.
\end{equation}
Let us denote
\begin{equation}
M=\sup\{ |\chi'(t)|+ |\chi''(t)|: 0\leq t \leq 1\}.
\end{equation}
If $ \frac{1}{2} (\tau_i(\delta))^2 \leq |z_i|^2 \leq (\tau_i(\delta))^2$, then
\begin{equation}
a_i \geq (2^{1-m_i} -cM)\delta^{-\frac{1}{m_i}}.
\end{equation}
Since $m_i$ and $M$ are independent of $\delta$, we can choose a small number
$c>0$, independent of $\delta$, to make $2^{1-m_i}-cM$ to be positive for all $i=1,\dots,n$. Letting
$d=\min_{1\leq i \leq n} \{2^{1-m_i}-cM\}$, we have
\begin{equation}\label{E:Middle}
a_i \geq d\delta^{-\frac{1}{m_i}}, \quad 1\leq i \leq n.
\end{equation}
Combining \eqref{E:Big}, \eqref{E:Small}, and \eqref{E:Middle}, we conclude that there
exists a constant $C>0$, independent of $\delta$, so that $a_i \geq C\deltanto{1}{m_i}$, $1\leq i \leq n$,
for
each small $\delta$.  Hence, we complete the proof.
\end{proof}
%%%
%%%
%%% Corollary for Iron Cross
%%%
%%%
\begin{corollary}\label{C:IronCross}
Let $v(z)$ be a plurisubharmonic function near the origin in $\mathbb C^n$ and let
$u(z)=v(z) + \sum_{i=1}^n|z_i|^{2m_i}$, where $m_i$'s are positive integers. If $\Omega$ be
a rigid domain with the mixed term $u(z)$ near origin in $\mathbb C^{n+1}$, then a subelliptic
estimate holds at the origin of order $\epsilon=\frac{1}{2m}$, where $m=\max_{1\leq i \leq n}m_i$.
\end{corollary}
%%%
%%%
%% corollary for monomial domains
%%%
%%%
\begin{corollary}\label{C:Monomial}
Let $f_1,\dots, f_l$ be  monomials in $z=(z_1,\dots,z_n)$ and let $\Omega$  be a pseudoconvex domain in $\mathbb
C^{n+1}$ whose boundary defining function near the origin is
\begin{equation}
2\Real (z_{n+1}) + \sum_{j=1}^l |f_j(z)|^2.
\end{equation}
If $T(b\Omega,0) <\infty$, then the best subelliptic estimate at the origin for $\Omega$ is
\begin{equation}
\epsilon=\frac{1}{T(b\Omega,0)},
\end{equation}
\end{corollary}
%%%
%%%
%%% Proof for monomials
%%%
%%%
\begin{proof}
Since $T(b\Omega,0) < \infty$ and $f_1,\dots,f_l$ are monomials, it follows that some positive power of each
coordinate $z_i$, $i=1,\dots,n$, equals one of  $f_1,\dots,f_l$ up to constant. Let $m_i$ be the smallest one
among such powers of $z_i$. Then, the mixed term of the boundary defining function near the origin satisfies
\eqref{E:AlmDiag}. Furthermore, one can show that $T(b\Omega,0)=\max\{m_i \mid i=1,\dots,n\}$. Therefore, Theorem
\ref{T:AlmDiag} and \cite[Theorem 3]{Ca83} implies the corollary.
\end{proof}

%%%
%%%
%%%
%%%

\bibliography{thesisbib}

\end{document}